\documentclass[12pt]{amsart}

\usepackage{amssymb,amsmath,latexsym,amsthm}
\usepackage[english]{babel}
\usepackage[all]{xy}

\newtheorem{theorem}{Theorem}[section]

\newtheorem{proposition}{Proposition}[section]
\newtheorem{cor}{Corollary}[section]

\theoremstyle{definition}

\newtheorem*{remark}{Remark}
\newtheorem*{example}{Example}

\begin{document}

\title{Roots of unity in definite quaternion orders}


\author{\sc Luis Arenas-Carmona}


\newcommand\Q{\mathbb Q}
\newcommand\alge{\mathfrak{A}}
\newcommand\Da{\mathfrak{D}}
\newcommand\Ea{\mathfrak{E}}
\newcommand\Ha{\mathfrak{H}}
\newcommand\oink{\mathcal O}
\newcommand\matrici{\mathbb{M}}
\newcommand\Txi{\lceil}
\newcommand\ad{\mathbb{A}}
\newcommand\enteri{\mathbb Z}
\newcommand\finitum{\mathbb{F}}
\newcommand\bbmatrix[4]{\left(\begin{array}{cc}#1&#2\\#3&#4\end{array}\right)}

\maketitle

\begin{abstract}
A commutative order in a quaternion algebra is called selective if it is embeds into some, but not all, 
the maximal orders in the algebra.  It is known that a given quadratic order over a number field can be selective 
in at most one indefinite quaternion algebra. Here we prove that the order generated by a cubic root of unity is 
selective for any definite quaternion algebra over the rationals with a type number 3 or larger.  The proof extends to a few  other closely related orders.

\end{abstract}

\bigskip
\section{Introduction}

A commutative order $\Ha$ in a quaternion algebra $\alge$, over a finite field $K$, is said to be selective if it is contained in some, but not all, of the maximal
orders in the algebra. More generally, for a genus $\mathbb{O}$ of orders of maximal rank in $\alge$, a commutative order
$\Ha$ is called selective for $\mathbb{O}$ if it embeds in some orders in $\mathbb{O}$, but not in all of them. The  selective
orders, for the genus of maximal order in a quaternion algebra, were characterized by Friedman and Chindburg \cite{FriedmannQ}, provided that the algebra satisfied Eichler's condition, namely:
\begin{quote} For any finite set $T$ of places of $K$ containing the set $\infty(K)$ of archimedean places, Eichler's condition for $T$ is satisfied if there exist a place $\wp\in T$ such that $\alge_\wp$ is a matrix algebra.
If $T$ is omitted, it is assumed that $T=\infty(K)$.
\end{quote}
 Later several authors extended this characterization to Eichler orders \cite{Guo}, \cite{Chan}. B. Linowitz 
has given several criteria under which selectivity can be avoided, always assuming  Eichler's condition 
\cite{lino}. All these result are based on
the fact that, if Eichler's condition holds, every spinor genus of orders of maximal rank in $K$ contains a unique 
conjugacy
class, and representations of orders by spinor genera can be studied by purely local computations, using the 
machinery of
class field theory. In fact, the proportion of spinor genera representing a given suborder (commutative or not), 
is
frequently a rational number of the form $1/[F:K]$, for some explicit class field $F$, the representation field. 
It is often the
case that $F\subseteq L$ when $\Ha$ is a suborder contained on a subfield $L$ of $\alge$. This shows that, for any 
such order $\Ha$, the proportion of spinor genera in a genus representing $\Ha$ is $0$, $1$, or $1/2$.  
The field $F$ is also contained in the spinor class field of the genus \cite{spinor}. This shows that, 
if the Eichler's condition is satisfied, for any given genus only a finite number
of subfields can contain selective orders. On the other hand, we showed in \cite{continuity} that any given commutative
order $\Ha$ can be selective in some genera of at most one quaternion algebra $\alge(\Ha)$. If $\Ha$ spans the field 
$L=K(\sqrt d)$,  then $\alge(\Ha)$ is defined by the Hilbert symbol $$ \alge(\Ha)=\left(\frac{d,-1}K\right).$$

The preceding discussion implies that, when Eichler's condition is satisfied, selectivity is a rare phenomenon. The
purpose of this work is to provide some evidence that this might not be so for quaternion algebras 
failing to satisfy Eichler's condition, i.e., definite quaternion algebras. 
Here we concentrate mostly on the case where $K=\mathbb{Q}$ and $\Ha=\enteri[\omega]$, where $\omega$
is a primitive cubic root of unity, but this is
due to the fact that this specific order has a particularly nice characterization in terms of quotient graphs.     

\begin{theorem}
Let $\mathbb{O}$ be a genus of Eichler orders of odd level in a definite rational quaternion algebra $\alge$ that is unramified at $2$. Assume that $\mathbb{O}$ contains $n$ conjugacy classes of orders, then at most $\frac n2+1$ of them contain cubic roots of unity.
\end{theorem}

\begin{cor}
Let $\mathbb{O}$ be a genus of Eichler orders in a definite rational quaternion algebra $\alge$ that is unramified at $2$.
Then if $\enteri[\omega]$ is represented by $\mathbb{O}$, it is selective for $\mathbb{O}$, as soon as
$\mathbb{O}$ contains at least three conjugacy classes.
\end{cor}

Note that the preceding corollary does not assume an odd level for the Eichler order. This is only due
to the fact that $\omega$ generates an unramified quadratic extension of $\Q_2$, and therefore
it is contained in a unique maximal order at $2$ \cite[Prop. 4.2]{Eichler2}, so an Eichler order of even
level cannot contain a copy of $\enteri[\omega]$.
The preceding result does not mention spinor genera, since every genus of Eichler orders in a quaternion 
algebra over $\mathbb{Q}$ has a unique spinor genus. This is not the case for more general orders.  

\begin{theorem}
Let $\mathbb{O}$ be a  spinor genus of orders of maximal rank  in a definite rational quaternion algebra $\alge$ that is unramified at $2$. Assume that the orders of $\mathbb{O}$ are maximal at $2$. Then, if $|\mathbb{O}|=n$, at most $\frac34(n+1)$ orders in $\mathbb{O}$ contain cubic roots of unity.
\end{theorem}

\begin{cor}
Let $\mathbb{O}$ be a spinor genus of orders of maximal rank  in a definite rational quaternion algebra $\alge$ that is unramified at $2$. Assume that the orders of $\mathbb{O}$ are maximal at $2$. Then if $\enteri[\omega]$ is represented by $\mathbb{O}$, it is selective for $\mathbb{O}$, as soon as
$\mathbb{O}$ contains at least four conjugacy classes.
\end{cor}

The latter result can be refined a little more using the theory of representation by spinor genera. An order
is called spinor selective for a given genus  if it is represented by some, but not all, the spinor genera in 
that genus. 

\begin{theorem}
Let $\mathbb{O}$ be a spinor genus of orders of maximal rank  in a definite rational quaternion algebra $\alge$
satisfying the following conditions:
\begin{itemize}
\item $\alge$ is unramified at $2$,
\item the orders in $\mathbb{O}$ are maximal at $2$,
\item $\mathbb{O}$ contains $3$ conjugacy classes, and 
\item $\enteri[\omega]$ is not spinor selective for the genus $\hat{\mathbb{O}}$ containing $\mathbb{O}$.
\end{itemize}
Then $\enteri[\omega]$ is selective for $\mathbb{O}$.
\end{theorem}

\section{Adeles and spinor genera}

Let $K$ be a number field, and let $\alge$ be a quaternion $K$-algebra. For every place $\wp$ of $K$,
we let $K_\wp$ and $\alge_\wp$ be the corresponding completions. We let $\mathbb{A}=\mathbb{A}_K$ be
the adele ring on $K$, namely the direct limit
$$ \mathbb{A}_K=\lim_{\longrightarrow}\oink_{\ad,T},\qquad \oink_{\ad,T}=
\left(\prod_{\wp\in T}K_\wp\right)\times\left(\prod_{\wp\in \Pi(K)-T}\oink_\wp\right),$$
where $\Pi(K)$ is the set of all places of $K$, 
$\oink_\wp$  is the ring of integers of the local field $K_\wp$, and the direct limit is taken with respect to the
directed set of finite subsets $T$ of $\Pi(K)$ containing $\infty(K)$. This defines a topology on
the ring $\mathbb{A}$. Similarly, if $\Da$ is an order of
maximal rank in $\alge$, we define, for every finite set $T\subseteq\Pi(K)$ containing $\infty(K)$, the $T$-adelization
$$ \Da_{\mathbb{A},T}=\left(\prod_{\wp\in T}\alge_\wp\right)\times\left(\prod_{\wp\in \Pi(K)-T}
\Da_\wp\right).$$  The adelization $\alge_{\mathbb{A}}$ is defined as the direct limit of the rings
$ \Da_{\mathbb{A},T}$. Note that this definition is independent of the choice of $\Da$ since two lattices on a 
$K$-vector space coincide at almost all places. The ring $\mathbb{A}$ is naturally identified with the sub-ring  of $\prod_{\wp\in\Pi(K)}K_\wp$
whose coordinates are integral at almost all places, although the topology of $\mathbb{A}$ as a direct limit is strictly
stronger than the product topology. The same observations apply to $\alge_{\mathbb{A}}$. For any order $\Da$, and
any finite set of places $T$ containing $\infty(K)$, the ring $\Da_{\mathbb{A},T}$ can be naturally identified with a subring
of $\alge_\ad$, and the induced topology is in fact the product topology of $\Da_{\mathbb{A},T}$. 

An order can be recovered from its adelization $\Da_{\mathbb{A},\infty(K)}$ by the formula
 $\Da=\Da_{\mathbb{A},\infty(K)}\cap\alge$, where $\alge$ embeds diagonally into $\alge_{\mathbb{A}}$.
More generally, for any finite set $T$ with $\infty(K)\subseteq T\subseteq\Pi(K)$, the intersection
$\Da^T=\Da_{\mathbb{A},T}\cap\alge$ is the $T$-order obtained from $\Da$ by inverting the non-archimedean
places in $T$.
Furthermore, conjugation induces a well defined action of $\alge_{\mathbb{A}}^*$, on the set of orders of maximal rank, 
satisfying $(a\Da a^{-1})_{\mathbb{A},T}=a\Da_{\mathbb{A},T}a^{-1}$, and therefore,
$(a\Da a^{-1})^T=a(\Da^T)a^{-1}$ for every finite set $T$ satisfying $\infty(K)\subseteq T\subseteq\Pi(K)$ \cite{abelianos}. Two orders $\Da$ and $\Da'$ of maximal rank in $\alge$ are in the same genus if and only if they 
are in the same orbit for this action, i.e., $\Da'=a\Da a^{-1}$ for some $a\in\alge_\ad$.
 They are in the same spinor genus when $a$ can be chosen as $a=bc$ where $b\in\alge$,
diagonally embedded into $\alge_{\mathbb{A}}$, while every local coordinate $c_\wp$ of $c$ has trivial reduced norm.
Genera and spinor genera for $T$-orders are defined analogously.
When $\alge$ satisfies Eichler's condition, every spinor genus contains a unique conjugacy class. This is not
 the case for definite quaternion algebras over the rationals, which are the ones that interest us in this work.
Nevertheless, spinor genera still play a role in the general setting (\S3).

\section{Spinor class fields and Classifying graphs}

For any genus $\mathbb{O}=\mathrm{Gen}(\Da)$ of orders of maximal rank, the spinor class field $\Sigma=\Sigma(\mathbb{O})$ is the class field corresponding to the class group
$K^*H(\Da)\subseteq J_K=\ad^*$, where the spinor image $H(\Da)$ is defined by
$$H(\Da)=\{N(a)|a\in\alge_{\mathbb{A}}^*,\ \Da=a\Da a^{-1}\},$$
where $N:\alge_{\mathbb{A}}^* \rightarrow J_K$ is the reduced norm on adeles. There is a well
defined \textit{distance} map $\rho:\mathbb{O}\times\mathbb{O}\rightarrow\mathrm{Gal}(\Sigma/K)$, satisfying
$\rho\big(\Da,a\Da a^{-1}\big)=[N(a),\Sigma/K]$, for any adelic element $a\in\alge_{\mathbb{A}}^*$,
 where $x\mapsto[x,\Sigma/K]$ is the Artin map on ideles.
Two orders $\Da$ and $\Da'$ are in the same spinor genus if and only if $\rho(\Da,\Da')=\mathrm{id}_\Sigma$.

More generally, for every finite set of places $T$ satisfying $\infty(K)\subseteq T\subseteq\Pi(K)$,
and for every genus of $T$-orders, we can define a spinor class field as above. In fact, for the genus
$\mathbb{O}^T=\{\Da^T|\Da\in\mathbb{O}\}$, the corresponding spinor class field $\Sigma^T$
is the largest subfield of $\Sigma=\Sigma(\mathbb{O})$ splitting completely at every non-archimedean place in $T$,
since the reduced norm is surjective at non-archimedean places.

Let $T=\infty(K)\cup\{\wp\}$ for some place $\wp$ splitting $\alge$. Then $\alge$ satisfies Eichler's condition for the
set $T$, so that every spinor genus of $T$-orders contains a unique conjugacy class. We conclude that, if $\Da_1$ and $\Da_2$ are in the same spinor genus, then the $T$-orders  $\Da^T_1$ and $\Da^T_2$ are conjugate. Replacing $\Da_2$ by a conjugate if needed, we can assume that $\Da^T_1=\Da^T_2$. In this case, $\Da_1$ and $\Da_2$ are conjugate if and only if  there is an element $a\in\alge$ satisfying both, $a\Da^T_1 a^{-1}=\Da^T_1$  and
$a\Da^T_{1,\wp} a^{-1}=\Da_{2,\wp}$. We conclude that conjugacy classes of orders $\Ha$, where $\Ha^T$ is
conjugate to $\Da^T$ are in one to one correspondence with the orbits of the conjugacy stabilizer 
$\Gamma=\mathrm{stab}_{\alge^*}(\Da^T)$ on the set of orders of maximal rank in $\alge_\wp$ that are
conjugate to $\Da_\wp$.

When $\Da_\wp$ is maximal, the latter set is simply the set $\mathbb{O}_\wp$ of maximal orders in $\alge_\wp$.
We can identify $\mathbb{O}_\wp$ with the set of vertices of the local Bruhat-Tits tree $\mathfrak{T}_\wp$ for
$\mathrm{PSL}_2(K_\wp)$ \cite[\S II.2]{vigneras}.
 In particular, the set of conjugacy classes of orders $\Ha$, where $\Ha^T$ is
conjugate to $\Da^T$ is in one to one correspondence with the set of vertices of the quotient graph (in the sense of Serre \cite{serre80}, see the remark bellow) $\Gamma\backslash\mathfrak{T}_\wp$. We call $\Gamma\backslash\mathfrak{T}_\wp$
the classifying graph of $\Da$ at $\wp$. Note that, when $\Da_\wp$ and $\Da'_\wp$ are
two neighbor in the tree, then $\Da'_\wp=a\Da_\wp a^{-1}$ for some local element $a\in\alge_\wp$ whose
reduced norm is a uniformizing parameter $\pi_\wp$.  In particular, if $\Da_1$ and $\Da_2$ are maximal orders
corresponding to neighboring vertices in the tree, their distance satisfies $\rho(\Da_1,\Da_2)=[e(\wp),\Sigma/K]$,
where $e(\wp)$ is an idele whose coordinates are as follows:
$$e(\wp)_{\mathfrak{q}}=\left\{\begin{array}{rcl}\pi_\wp&\textnormal{ if }&\mathfrak{q}=\wp\\
1&\textnormal{ if }&\mathfrak{q}\neq\wp\end{array}\right..$$
We conclude that the orders corresponding to the vertices in the classifying graph of $\Da$ belong to a unique
spinor genus if $\wp$ splits in $\Sigma/K$ and two spinor genera otherwise. In the latter case the classifying graph is bipartite, so in particular, it contains no loops. 

\begin{remark}
The definition of quotient graph given in \cite{serre80} assumes that the group acts on the tree without inversion.
This can be fixed, as noted in the same reference, by replacing the tree by its first barycentric subdivision. We adopt
this convention in all that follows, but we reserve the word vertex for a vertex of the original graph, 
while the barycenters
of edges are called virtual vertices. When drawing actual pictures, virtual 
vertices are not drawn unless they are  endpoints,
i.e., their valency is $1$ (See Fig. 1).
This type of endpoints appear only if the original action had inversions.  
\begin{figure}
\unitlength 1mm 
\linethickness{0.4pt}
\ifx\plotpoint\undefined\newsavebox{\plotpoint}\fi 
\[\textnormal{(A)}\ 
\begin{picture}(24,6)(0,0)
\put(2,3){\line(1,0){12}}\put(2,3){\makebox(0,0)[cc]{$\bullet$}}\put(14,3){\makebox(0,0)[cc]{$*$}}
\end{picture}
\]\caption{ A vertex ($\bullet$) with an edge ending on a virtual endpoint ($*$). }
\end{figure}
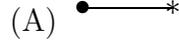
\end{remark}
\begin{remark}
We should note that the quotient graph defined here have been applied in existing literature to the study
of representation of commutative orders. See for example \cite{pays}.
\end{remark}

\section{Characterizing the orders with cubic roots}

Next result gives a complete characterization of the orders in a genus containing a cubic root of unity. This will be
used in next section to prove our main theorems.

\begin{proposition}
Let $\mathbb{O}$ be a genus of orders of maximal rank in the definite rational quaternion algebra $\alge$ containing a cubic
root of unity. Assume the following conditions hold:
\begin{itemize}
\item $\alge_2\cong\matrici_2(\mathbb{Q}_2)$.
\item The orders in  $\mathbb{O}$ are maximal at $2$.
\end{itemize}
Let $\mathfrak{G}=\Gamma\backslash\mathfrak{T}_2$ be the classifying graph at $2$, 
as described in the preceding section.
Then the orders orders in  $\mathbb{O}$ containing a cubic root of unity correspond 
precisely to the (non-virtual) endpoints, i.e., vertices of valency one, in $\mathfrak{G}$.
\end{proposition}

\begin{proof}
 Let $\Da\in\mathbb{O}$ be an order containing a cubic root of unity $\omega$. Since $\enteri_2[\omega]$ is
the ring of integers in the unique unramified quadratic extension of $\mathbb{Q}_2$, then $\enteri_2[\omega]$
 is contained in a unique maximal order of $\alge_2$ \cite[Prop. 4.2]{Eichler2}, namely $\Da_2$.
 Recall that the conjugation-stabilizer of
a local maximal order $\Ea_2$, is $K^*\Ea_2^*$. Since $\omega$ is a unit we conclude that conjugation by $\omega$ 
permutes transitively the three neighbors of $\Da$, and therefore  the vertex in $\mathfrak{G}$ corresponding to $\Da$
has valency one. 

Now let $\Da$ be a global order corresponding to a vertex of valency one in the classifying graph $\mathfrak{G}$.
Let $\Gamma_{\Da}$ be the stabilizer of $\Da$ in $\Gamma$. Then $\overline{\Gamma_{\Da}}=
\Gamma_{\Da}/K^*$ is the automorphism group of the order
$\Da$, and therefore it is finite, since $\alge$ is definite. 
Since $\Da$  corresponds to a vertex of valency one in $\mathfrak{G}$, the
group $\overline{\Gamma_{\Da}}$ must permute transitively the three neighbors of $\Da$, and therefore it contains an
element of order $3$. In other words, there exists an element $u\in \alge\backslash K$ satisfying $u\Da u^{-1}=\Da$ and $u^3\in K^*$. This element has a quadratic minimal polynomial with rational (and therefore real) coefficients. So
it has either two real roots or two conjugate complex roots.

Now consider $u$ as an element of $\alge_{\bar{\mathbb{Q}}}=\alge\otimes_{\mathbb{Q}}\bar{\mathbb{Q}}\cong\matrici_2(\bar{\mathbb{Q}})$, where
$\bar{\mathbb{Q}}$ is the field of algebraic numbers.
Then $u$ is conjugate to a real multiple of a matrix of the form $\bbmatrix{\eta_1}00{\eta_2}$ where $\eta_1$ and
$\eta_2$ are two different cubic roots of unity. By the previous discussion we can 
assume either $(\eta_1,\eta_2)=
(\omega,\omega^2)$ or $(\eta_2,\eta_1)=(\omega,\omega^2)$. Hence, the projective representation $\rho:C_3\rightarrow \alge^*/K^*$
of the cyclic group $C_3$ defined by $u$ is conjugated to the representation defined by
a cubic root of unity over $\overline{\mathbb{Q}}$, and therefore also over $\mathbb{Q}$ 
by \cite[Th. 1]{proj}, so we can assume that $u$ is a root of unity. Since the stabilizer of a local maximal order
$\Da_\wp$ equals $K_\wp^*\Da_\wp^*$ at split places, and $\Da_\wp$ contains all local integers at ramified places,
we must conclude that $u\in\Da^*$, locally everywhere and hence globally. The result follows.
\end{proof}

\begin{example}
Consider the quaternion algebra $\alge=\left(\frac{-3,-3}{\mathbb{Q}}\right)$. Let $i$ and $j$ be generators
satisfying $i^2=j^2=-3$ and $ij=-ji$. Set $i=2\eta+1$ and $j=2\omega+1$, so that $\eta$ and $\omega$ are
cubic roots of unity. Then each of the orders $\Da=\enteri[\eta,j]$ and $\Da'=\enteri[i,\omega]$ is contained in a unique maximal order, since they are maximal outside the place $3$ and $\alge_3$ has a unique maximal order. Denote these maximal orders by $\Da_0$ and $\Da'_0$. They  are endpoints of the Classifying
Graph at $2$ by the previous proposition. We claim that they are neighbors.
 Since they are also isomorphic, and therefore conjugate, we must conclude that 
all maximal orders in $\alge$ are conjugate.

Now we prove the claim. It follow from the tables in \cite{isaa} that $\enteri_2[i,j]$ is contained in exactly $2$
maximal orders. Alternatively, we can observe that $\enteri_2[i+j]\cong\enteri_2[\sqrt{-6}]$ is the ring of 
integers of a ramified extension, and therefore it is contained in exactly $2$ maximal orders by \cite[Prop. 4.2]{Eichler2}.
\end{example}

\section{Proofs of the main results}

Theorem 1 follows from proposition 4.1 and Proposition \ref{p51} below:

\begin{proposition}\label{p51}
Let $G$ be a connected graph where every vertex has valency $3$ or less. If $G$ has $n=r+t$ vertices, $r$ of which
have valency $1$, then $r\leq t+2$, or equivalently, $r\leq\frac n2+1$. Equality, in both inequalities, 
holds if and only if
$G$ is a tree where every vertex have valency $1$ or $3$.
\end{proposition}

\begin{proof}
Since the graph is connected, the number of edges is $v\geq n-1$, with equality if and only if the graph is a tree.
 On the other hand,  since every edge has two endpoints, $3t+r\geq 2v$. We conclude that $3t+r\geq 2(t+r-1)$,
with equality if and only if each of the previous inequalities is an equality. The result follows.
\end{proof}

\begin{remark}
Note that we might not have equality in Theorem 1.1 even if the quotient graph satisfy the sufficient conditions
for equality in Proposition \ref{p51} above, because of the existence of virtual endpoints. This can be used in a few cases to improve the bound $r\leq t+2$ and, a fortiori, also $r\leq\frac n2+1$. In fact, the presence of some particular subfields imply the existence of virtual endpoints. See the examples in the
final section.
\end{remark}

Next result is needed in the proof of Theorem 2.

\begin{proposition}\label{p52}
Let $G$ be a bipartite graph with a set of vertices
$V(G)=A\cup B$, where every vertex of $G$ joins a vertex in $A$ and a vertex in $B$.
Let $n$ be the cardinality of $B$, while $r$ is the number of endpoints in $B$. If no vertex of $G$ has a valency larger
than $3$, then $r\leq3(n+1)/4$. Equality holds if and only if  the following conditions hold:
\begin{enumerate}
\item $G$ is a tree.
\item Every vertex in $A$ has valency $3$.
\item Every vertex in $B$ has valency $3$ or $1$.
\end{enumerate}
\end{proposition}

\begin{proof}
First we assume that $G$  satisfies conditions (1)-(3).
Let $t=n-r$ as before.
Let $m$, $p$, and $s$ the numbers of vertices in $A$ with zero, one, or two neighboring endpoints, respectively.
Then the previous result shows that $r=t+m+p+s+2$, while on the other hand, we have the identities
$r= p+2s$ and $3t=3m+2p+s$. Adding this two identities, we get $r+3t=3(m+p+s)=3(r-t-2)$, whence
$r=3(t+1)$. Replacing $t$ by $n-r$, the result follows.

We call nails and forks to subgraphs of the shapes shown in Figure 1(A). Now, for the general case, we observe that,
by adding extra nails and forks, as shown in Figure 1(C),
 we can turn any tree in a tree satisfying the preceding hypotheses. Furthermore, cycles
can be eliminated by replacing an edge by a nail-fork pair, as shown in Figure 1(B). By repeated application of these procedures, we obtain a graph $G'$ satisfying the hypotheses. Set $V(G')=A'\cup B'$ as before, where $B$ is
identified with a subset of $B'$ in an obvious way. Let $n'$ be the cardinality of $B'$, while $r'$ is the number of endpoints 
in $B'$, and $t'=n'-r'$. Note that $t'=t$ and $r\leq r'$, by a case by case inspection of the reduction steps in Figure 1. 
We conclude that $$r\leq r'=3(t'+1)=3(t+1)=3(n+1)-3r,$$
whence the result follows. Furthermore, the inequality is strict unless no reduction was needed.

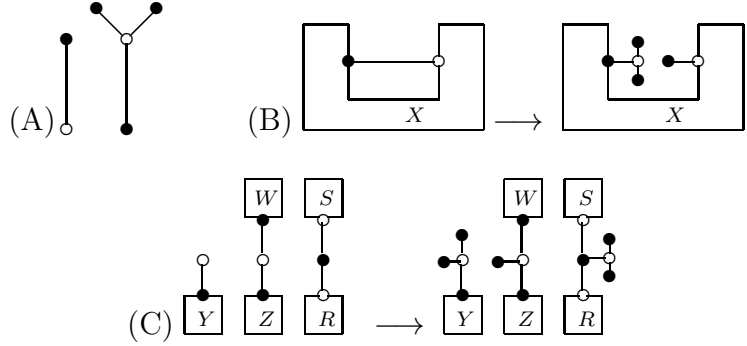
\begin{figure}
\unitlength 1mm 
\linethickness{0.4pt}
\ifx\plotpoint\undefined\newsavebox{\plotpoint}\fi 
\[\textnormal{(A)}\ 
\begin{picture}(24,24)(0,0)
\put(0,0.7){\line(0,1){11.3}}
\put(8,0){\line(0,1){11.4}}\put(8.5,12.5){\line(1,1){3.7}}\put(7.5,12.5){\line(-1,1){3.7}}
\put(0,0){\makebox(0,0)[cc]{$\circ$}}\put(0,12){\makebox(0,0)[cc]{$\bullet$}}
\put(8,0){\makebox(0,0)[cc]{$\bullet$}}
\put(8,12){\makebox(0,0)[cc]{$\circ$}}
\put(12,16){\makebox(0,0)[cc]{$\bullet$}}\put(4,16){\makebox(0,0)[cc]{$\bullet$}}
\end{picture}
\textnormal{(B)}\ 
\begin{picture}(24,24)(0,0)
\put(0,0){\line(0,1){14}}\put(0,0){\line(1,0){24}}\put(24,0){\line(0,1){14}}\put(0,14){\line(1,0){6}}
\put(6,14){\line(0,-1){10}}\put(6,4){\line(1,0){12}}\put(18,4){\line(0,1){4.5}}
\put(18,9.7){\line(0,1){4.3}}\put(18,14){\line(1,0){6}}
\put(6,9){\makebox(0,0)[cc]{$\bullet$}}\put(18,9){\makebox(0,0)[cc]{$\circ$}}
\put(6,9){\line(1,0){11.4}}
\put(15,2){\makebox(0,0)[cc]{${}_{X}$}}
\end{picture}
\longrightarrow\ 
\begin{picture}(24,24)(0,0)
\put(0,0){\line(0,1){14}}\put(0,0){\line(1,0){24}}\put(24,0){\line(0,1){14}}\put(0,14){\line(1,0){6}}
\put(6,14){\line(0,-1){10}}\put(6,4){\line(1,0){12}}\put(18,14){\line(1,0){6}}\put(18,4){\line(0,1){4.5}}
\put(18,9.7){\line(0,1){4.3}}
\put(6,9){\makebox(0,0)[cc]{$\bullet$}}\put(18,9){\makebox(0,0)[cc]{$\circ$}}
\put(6,9){\line(1,0){3.4}}\put(10,9){\makebox(0,0)[cc]{$\circ$}}\put(10,9.7){\line(0,1){2}}\put(10,8.5){\line(0,-1){2}}
\put(10,11.5){\makebox(0,0)[cc]{$\bullet$}}\put(10,6.5){\makebox(0,0)[cc]{$\bullet$}}
\put(17.4,9){\line(-1,0){3.4}}\put(14,9){\makebox(0,0)[cc]{$\bullet$}}
\put(15,2){\makebox(0,0)[cc]{${}_{X}$}}
\end{picture}
\]
\[
\textnormal{(C)}\ 
\begin{picture}(24,24)(0,0)
\put(0,0){\line(0,1){5}}\put(0,5){\line(1,0){5}}\put(0,0){\line(1,0){5}}\put(5,0){\line(0,1){5}}
\put(2.5,5){\makebox(0,0)[cc]{$\bullet$}}\put(2.3,5.2){\line(0,1){4}}\put(2.5,9.7){\makebox(0,0)[cc]{$\circ$}}
\put(3,2){\makebox(0,0)[cc]{${}_{Y}$}}
\put(8,0){\line(0,1){5}}\put(8,5){\line(1,0){5}}\put(8,0){\line(1,0){5}}\put(13,0){\line(0,1){5}}
\put(10.5,5){\makebox(0,0)[cc]{$\bullet$}}\put(10.3,5.2){\line(0,1){4}}\put(10.5,9.7){\makebox(0,0)[cc]{$\circ$}}
\put(11,2){\makebox(0,0)[cc]{${}_{Z}$}}\put(8,20.5){\line(0,-1){5}}\put(8,15.5){\line(1,0){5}}\put(8,20.5){\line(1,0){5}}\put(13,20.5){\line(0,-1){5}}
\put(10.5,15){\makebox(0,0)[cc]{$\bullet$}}\put(10.3,14.8){\line(0,-1){4}}
\put(11,18){\makebox(0,0)[cc]{${}_{W}$}}
\put(16,0){\line(0,1){5}}\put(16,5){\line(1,0){2}}\put(19,5){\line(1,0){2}}
\put(16,0){\line(1,0){5}}\put(21,0){\line(0,1){5}}
\put(18.5,5){\makebox(0,0)[cc]{$\circ$}}
\put(18.3,5.7){\line(0,1){4}}\put(18.5,9.7){\makebox(0,0)[cc]{$\bullet$}}
\put(19,2){\makebox(0,0)[cc]{${}_{R}$}}\put(16,20.5){\line(0,-1){5}}
\put(19,15.5){\line(1,0){2}}\put(16,15.5){\line(1,0){2}}\put(16,20.5){\line(1,0){5}}\put(21,20.5){\line(0,-1){5}}
\put(18.5,15){\makebox(0,0)[cc]{$\circ$}}\put(18.3,14.8){\line(0,-1){4}}
\put(19,18){\makebox(0,0)[cc]{${}_{S}$}}
\end{picture}
\longrightarrow\ 
\begin{picture}(24,24)(0,0)
\put(0,0){\line(0,1){5}}\put(0,5){\line(1,0){5}}\put(0,0){\line(1,0){5}}\put(5,0){\line(0,1){5}}
\put(2.5,5){\makebox(0,0)[cc]{$\bullet$}}\put(2.3,5.2){\line(0,1){4}}\put(2.5,9.7){\makebox(0,0)[cc]{$\circ$}}
\put(3,2){\makebox(0,0)[cc]{${}_{Y}$}}
\put(8,0){\line(0,1){5}}\put(8,5){\line(1,0){5}}\put(8,0){\line(1,0){5}}\put(13,0){\line(0,1){5}}
\put(10.5,5){\makebox(0,0)[cc]{$\bullet$}}\put(10.3,5.2){\line(0,1){4}}\put(10.5,9.7){\makebox(0,0)[cc]{$\circ$}}
\put(11,2){\makebox(0,0)[cc]{${}_{Z}$}}\put(8,20.5){\line(0,-1){5}}\put(8,15.5){\line(1,0){5}}\put(8,20.5){\line(1,0){5}}\put(13,20.5){\line(0,-1){5}}
\put(10.5,15){\makebox(0,0)[cc]{$\bullet$}}\put(10.3,14.8){\line(0,-1){4}}
\put(11,18){\makebox(0,0)[cc]{${}_{W}$}}
\put(16,0){\line(0,1){5}}\put(16,5){\line(1,0){2}}\put(19,5){\line(1,0){2}}
\put(16,0){\line(1,0){5}}\put(21,0){\line(0,1){5}}
\put(18.5,5){\makebox(0,0)[cc]{$\circ$}}
\put(18.3,5.7){\line(0,1){4}}\put(18.5,9.7){\makebox(0,0)[cc]{$\bullet$}}
\put(19,2){\makebox(0,0)[cc]{${}_{R}$}}\put(16,20.5){\line(0,-1){5}}
\put(19,15.5){\line(1,0){2}}\put(16,15.5){\line(1,0){2}}\put(16,20.5){\line(1,0){5}}\put(21,20.5){\line(0,-1){5}}
\put(18.5,15){\makebox(0,0)[cc]{$\circ$}}\put(18.3,14.8){\line(0,-1){4}}
\put(19,18){\makebox(0,0)[cc]{${}_{S}$}}
\put(18,10){\line(1,0){3.4}}\put(22,10){\makebox(0,0)[cc]{$\circ$}}\put(22,10.7){\line(0,1){2}}
\put(22,9.5){\line(0,-1){2}}
\put(22,12.5){\makebox(0,0)[cc]{$\bullet$}}\put(22,7.5){\makebox(0,0)[cc]{$\bullet$}}
\put(10.1,9.6){\line(-1,0){2.2}}\put(7.2,9.4){\makebox(0,0)[cc]{$\bullet$}}
\put(2.1,9.6){\line(-1,0){2}}\put(0,9.4){\makebox(0,0)[cc]{$\bullet$}}
\put(2.4,10.8){\line(0,1){2}}\put(2.4,13){\makebox(0,0)[cc]{$\bullet$}}
\end{picture}
\] \caption{(A) Nail and fork. (B) Deleting a cycle by introducing a nail-fork pair.
(C) Adding nails and forks to several trees.}
\end{figure}
\end{proof}

\subparagraph{\textit{Proof of Theorem 2.}}
 Let $\Da$ be an order in the spinor genus $\mathbb{O}$, and let $\Sigma$ be the spinor class field 
corresponding to this genus. Let $e(2)\in J_{\Q}$ be the idele whose only non-trivial coordinate is $e(2)_2=2$,
and let $\sigma=[e(2),\Sigma/\Q]$. If $\sigma$ is the identity, the vertices in the quotient graph $\mathfrak{G}$
correspond to maximal orders in one spinor genus and the result follows from Proposition \ref{p51} since
$\frac n2+1\leq \frac34(n+1)$ for $n\geq1$. If
$\sigma$ is not the identity, the graph $\mathfrak{G}$ is bipartite, and its vertices belong to exactly two spinor genera, so the result follows from Proposition \ref{p52}. \qed

\subparagraph{\textit{Proof of Theorem 3.}}
Assume all conditions in the theorem hold, and let $\sigma$ be as in the preceding proof. If $\sigma$ is the
identity, the vertices in $\mathfrak{G}$ correspond to maximal orders in one spinor genus and the result follows from Proposition \ref{p51}  as before. We assume therefore that $\sigma$ is not the identity, so in particular the graph $\mathfrak{G}$ is bipartite, with any pair of adjacent vertices  in different spinor genera. Let $\mathbb{O}$ and
$\mathbb{O}_1$ be such spinor genera. Since $\enteri[\omega]$ is not spinor selective for this genus, there must
exists an order in each spinor genus containing a copy of $\enteri[\omega]$. Let $\Da_1\in\mathbb{O}_1$
 be such an order. Let $\Da\in\mathbb{O}$ be a neighbor of $\Da_1$. If $\enteri[\omega]$ is not selective for the
spinor genus $\mathbb{O}$, both $\Da$ and $\Da_1$ are endpoints and the graph contains only two vertices,
a contradiction.\qed

\section{Generalizations and examples}

The methods employed here for the ring $\enteri[\omega]$ works with a limited number of orders whose presence
is reflected on the combinatorial properties of the local graph at some finite place $\wp$. 

\begin{example}
Roots
of unity $\eta$ whose order is a Fermat prime $p=2^{2^n}+1$ produce vertices of valency $1$ for any definite
quaternion algebra defined over the totally real field $L=\mathbb{Q}(\eta)\cap\mathbb{R}$, at any dyadic place 
$\wp$ of $L$. This can be proved as follows:
\begin{quote}
The extension
$\mathbb{Q}(\eta)/\mathbb{Q}$ is unramified at $2$, since the polynomial $x^p-1$
has different roots in $\overline{\mathbb{F}_2}$, while each of them generates an extension of degree
$2^n+1$ over $\mathbb{F}_2$, as $m=2^n+1$ is the smallest value of $m$ for which $p$ divides $2^m-1$. 
Furthermore, the quadratic extension 
$\mathbb{Q}(\eta)/L$ is inert at every dyadic place, since $\eta\mapsto\eta^{-1}$ induces a non-trivial 
isomorphism over the residue field.
 In particular, for every dyadic place $\wp$ of $L$, the residue field
$\mathbb{L}$ satisfies $[\mathbb{L}:\mathbb{F}_2]=2^n$. We conclude that the local Bruhat-Tits tree for 
$\mathrm{PSL}_2(L_\wp)$ at any such place has vertices of valency $2^{2^n}+1=p$. Now the result follows
as in the proof of Proposition 4.1.
\end{quote}

As a consequence, we conclude, as in the proof of Theorem 1.2 that the number $r$ of conjugacy classes,
in any spinor genus of $\oink_L$-orders of size $m$, containing a unit of order $p=2^{2^n}+1$, satisfies
$$ r\leq\frac{p(p-2)}{(p-1)^2}\left(m+\frac1{p-2}\right).$$
This again shows that $\oink_L[\eta]$ is selective for large values of $m$.
\end{example}

\begin{example}
Let $\alge$ be a definite quaternion algebra splitting at $2$.
Consider the commutative order $\enteri[u]=\oink_{\mathbb{Q}[u]}$, where $u$ is a root of the equation
$x^2-x+2=0$. Note that the extension $\mathbb{Q}[u]/\mathbb{Q}$ splits at $2$, whence the local orders
at $2$ containing a fixed copy of  $\enteri_2[u]$ in $\alge_2$ lie in a maximal path $\gamma$ in the Bruhat-Tits tree  
\cite[Prop. 4.2]{Eichler2}. Since the norm of $u$ is a uniformizing parameter, it is easy to see that,
 for any global embedding $\phi:\mathbb{Q}[u]\rightarrow\alge$,
conjugation by $\phi(u)$ defines a global automorphism of $\alge$ that shifts by $1$ the line of local orders
containing $\phi(u)$. Since $u$ is a unit outside of $2$, conjugation by $\phi(u)$ must stabilize any 
$\mathbb{Z}[1/2]$-order $\Da$ containing it. It follows that all orders in $\gamma$ are isomorphic and are 
connected to equivalent branches of the tree (Figure 3A).
We conclude that the corresponding vertex in the quotient graph looks like one of the graphs in Figure 3B, 
according to
whether there exist a global inversion on this path or not.
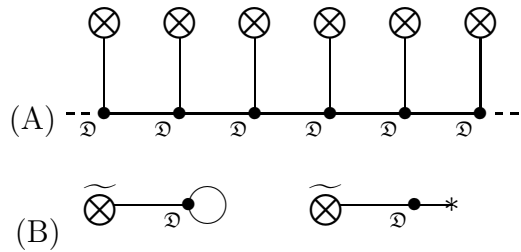
\begin{figure}
\unitlength 1mm 
\linethickness{0.4pt}
\ifx\plotpoint\undefined\newsavebox{\plotpoint}\fi 
\[\textnormal{(A)}\ 
\begin{picture}(70,24)(0,0)
\put(5,2){\line(1,0){50}}\put(0,2){\line(1,0){1}}\put(2,2){\line(1,0){1}}\put(57,2){\line(1,0){1}}\put(59,2){\line(1,0){1}}
\put(5,2){\makebox(0,0)[cc]{$\bullet$}}\put(5,2){\line(0,1){10}}\put(5,14){\makebox(0,0)[cc]{$\bigotimes$}}
\put(15,2){\makebox(0,0)[cc]{$\bullet$}}\put(15,2){\line(0,1){10}}\put(15,14){\makebox(0,0)[cc]{$\bigotimes$}}
\put(25,2){\makebox(0,0)[cc]{$\bullet$}}\put(25,2){\line(0,1){10}}\put(25,14){\makebox(0,0)[cc]{$\bigotimes$}}
\put(35,2){\makebox(0,0)[cc]{$\bullet$}}\put(35,2){\line(0,1){10}}\put(35,14){\makebox(0,0)[cc]{$\bigotimes$}}
\put(45,2){\makebox(0,0)[cc]{$\bullet$}}\put(45,2){\line(0,1){10}}\put(45,14){\makebox(0,0)[cc]{$\bigotimes$}}
\put(55,2){\makebox(0,0)[cc]{$\bullet$}}\put(55,2){\line(0,1){10}}\put(55,14){\makebox(0,0)[cc]{$\bigotimes$}}
\put(3,0){\makebox(0,0)[cc]{${}_\Da$}}\put(13,0){\makebox(0,0)[cc]{${}_\Da$}}
\put(23,0){\makebox(0,0)[cc]{${}_\Da$}}\put(33,0){\makebox(0,0)[cc]{${}_\Da$}}
\put(43,0){\makebox(0,0)[cc]{${}_\Da$}}\put(53,0){\makebox(0,0)[cc]{${}_\Da$}}
\end{picture}
\]
\[\textnormal{(B)}
\begin{picture}(70,12)(0,0)
\put(5.2,5){\makebox(0,0)[cc]{$\widetilde{\bigotimes}$}}
\put(7,5){\line(1,0){10}}\put(17,5){\makebox(0,0)[cc]{$\bullet$}}
\put(19.5,5){\circle{5}}
\put(35.2,5){\makebox(0,0)[cc]{$\widetilde{\bigotimes}$}}
\put(37,5){\line(1,0){10}}\put(47,5){\makebox(0,0)[cc]{$\bullet$}}
\put(47,5){\line(1,0){5}}\put(52,5){\makebox(0,0)[cc]{$*$}}

\put(15,3){\makebox(0,0)[cc]{${}_\Da$}}\put(45,3){\makebox(0,0)[cc]{${}_\Da$}}
\end{picture}
\] \caption{(A) Repetitive pattern arising from a global 
uniformizing parameter in a field splitting at $2$. (B) Possible shapes of the
quotient graph. Here $\widetilde{\bigotimes}$ denotes the image of the subgraph $\bigotimes$.}
\end{figure}
 It follows that the proof of Theorems 1.1-1.3 carry
over word by word to this case. 
\end{example}

\begin{example}
When $K$ is the function field of a smooth projective curve $X$ over a finite field $\mathbb{F}$, quotient graphs similar
to those described here can be used to classify $X$-orders in a quaternion $K$-algebra $\alge$ \cite[Thm. 1.1]{qgraphs}.
 These graphs are finite if and only if $\alge$ is a division algebra. They depend on the choice of a place
at infinity, playing the same role as the place  $2$ in our case. If the place at infinity is defined over $\mathbb{F}$,
 endpoints correspond to orders representing the maximal order of the subfield
$K\mathbb{L}$ where $\mathbb{L}$ is the unique quadratic extension of $\mathbb{F}$ \cite[Ex. 5.1]{qgraphs}.
The same argument given here can be used to show that this order is selective on large spinor genera.
\end{example}

\begin{example}
It follows easily from \cite[Prop. 2.4]{Eichler2} that if the algebra $\alge$ contains a
cubic root of unity and if $r$ is the largest distance of a vertex from the set of non-virtual endpoints, 
in the classifying graph at $2$ of maximal orders in $\alge$,
then $\Ha=\enteri+2^r\enteri[\omega]=\enteri[2^{r-1}\sqrt{-3}]$ is contained in every maximal order
in $\alge$. 
The methods in this work cannot be used to prove selectivity for the order $\enteri[\sqrt{-3}]$, since it is
easy to draw graphs where every vertex is at distance $1$ from an endpoint, as in Figure 4.
\begin{figure}
\unitlength 1mm 
\linethickness{0.4pt}
\ifx\plotpoint\undefined\newsavebox{\plotpoint}\fi 
\[
\begin{picture}(60,12)(0,0)
\put(0,5){\makebox(0,0)[cc]{$\bullet$}}
\put(5,5){\makebox(0,0)[cc]{$\bullet$}}
\put(5,10){\makebox(0,0)[cc]{$\bullet$}}
\put(0,5){\line(1,0){5}}\put(5,5){\line(0,1){5}}
\put(10,5){\makebox(0,0)[cc]{$\bullet$}}
\put(10,10){\makebox(0,0)[cc]{$\bullet$}}
\put(5,5){\line(1,0){5}}\put(10,5){\line(0,1){5}}
\put(15,5){\makebox(0,0)[cc]{$\bullet$}}
\put(15,10){\makebox(0,0)[cc]{$\bullet$}}
\put(10,5){\line(1,0){5}}\put(15,5){\line(0,1){5}}
\multiput(15,5)(1,0){20}{.}
\put(35,5){\makebox(0,0)[cc]{$\bullet$}}
\put(35,10){\makebox(0,0)[cc]{$\bullet$}}
\put(35,5){\line(1,0){5}}\put(35,5){\line(0,1){5}}
\put(40,5){\makebox(0,0)[cc]{$\bullet$}}
\put(40,10){\makebox(0,0)[cc]{$\bullet$}}
\put(40,5){\line(1,0){5}}\put(40,5){\line(0,1){5}}
\put(45,5){\makebox(0,0)[cc]{$\bullet$}}
\put(45,10){\makebox(0,0)[cc]{$\bullet$}}
\put(45,5){\line(1,0){5}}\put(45,5){\line(0,1){5}}
\put(50,5){\makebox(0,0)[cc]{$\bullet$}}
\end{picture}
\] 
\caption{The classifying graph of a conjectural large genus for which $\enteri[\sqrt{-3}]$ 
would not be selective.}
\end{figure}
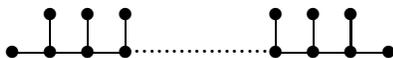
 Whether these graphs can actually
be classifying graphs for some family of spinor genera is still unknown to us. If such genera existed we would
have an example of an order that is non-selective for arbitrarily large spinor genera.

By the theorem of almost strong approximation \cite[Th. 3.1]{ae}, 
the order $\enteri[p^2\sqrt{-d}]$ is not selective in a fixed genus,
for almost every prime $p$, since this result implies that every pair of maximal orders are 
at most at a distance $2$ in the corresponding graph. Furthermore, the main result in
\cite{earnest} states:
\begin{quote}
Any positive definite ternary quadratic lattice of discriminant $d$ represents any integer of the form $ct^2$,
if $t$ is large enough, $(t,2d)=1$, and $c$ is primitively represented by the corresponding genus.
\end{quote}
This certainly implies a stronger version of the previous statements. All these results refer to large factors,
however, so the following question remains open, as far as we know: 
\begin{quote}
\textit{Is any order in a imaginary quadratic field selective for any large enough spinor genus?}
\end{quote}
As noted above, if the answer to this question is positive, this has significant implications in the possible shapes
 of the classifying graphs.
\end{example}

Let $\alge(p)$ be the unique quaternion algebra over $\mathbb{Q}$ that ramifies exactly at the places 
$p$ and $\infty$. We end this work by computing the quotient graph at $2$ of the genus of maximal orders
for all odd primes for which the class number of the algebra  $\alge(p)$ is $1$, that is when $p\in\{3,5,7,13\}$,
cf. \cite[\S3]{Ichi}.

When $\mathbb{Q}[\omega]$ embeds into $\alge(p)$, namely, when $p\in\{3,5\}$, the unique vertex
of the classifying graph has valency $1$, an therefore the graph looks like Figure 5A. Since $\mathbb{Q}[\sqrt{-7}]$
embeds into $\alge(7)$, it follows from the second example in this section that  the classifying graph
$\Gamma\backslash\mathfrak{T}_2$ has one of the shapes
in Figure 3B. Note that $\alge(7)=\left(\frac{-7,-1}{\mathbb{Q}}\right)$, whence for any embedding $\phi:\mathbb{Q}[u]\rightarrow\alge(7)$, where $u=\frac{\sqrt{-7}+1}2$, there in a pure quaternion $i\in\alge(7)$ satisfying $i^2=-1$ and $i\phi(u)i^{-1}=\overline{\phi(u)}$, whence conjugation by $i$ permutes the eigenvectors of $\phi(u)$.
 Since conjugation by $\phi(u)$ is a shift in the path $\gamma$ of maximal orders containing $\phi(u)$, as
in Figure 3A,
conjugation by $i$ is an inversion on that line. We conclude that $\Gamma\backslash\mathfrak{T}_2$
 looks like Figure 5B.
The same holds for $p=13$ if we note that $\alge(13)=\left(\frac{-7,-13}{\mathbb{Q}}\right)$,
and any pure quaternion satisfying $i^2=-13$ stabilizes any maximal $\{\infty,2\}$-order containing it, even if $i$ is not a $\{\infty,2\}$-unit,
since the maximal order at $13$ is unique, and $N(i)$ is a unit outside $13$. The result follows. 
\begin{figure}
\unitlength 1mm 
\linethickness{0.4pt}
\ifx\plotpoint\undefined\newsavebox{\plotpoint}\fi 
\[\textnormal{(A)}
\begin{picture}(35,12)(0,0)
\put(10,5){\makebox(0,0)[cc]{$\bullet$}}
\put(10,5){\line(1,0){10}}\put(20,5){\makebox(0,0)[cc]{$*$}}
\end{picture}
\textnormal{(B)}
\begin{picture}(35,12)(0,0)
\put(11,5){\makebox(0,0)[cc]{$\bullet$}}
\put(11,5){\line(1,0){9}}\put(20,5){\makebox(0,0)[cc]{$*$}}
\put(2,5){\line(1,0){9}}\put(2,5){\makebox(0,0)[cc]{$*$}}
\end{picture}
\] \caption{ Classifying graphs for the genus of 
maximal orders at $2$ for the algebra $\alge(p)$ when (A) $p\in\{3,5\}$. (B) $p\in\{7,13\}$.}
\end{figure}
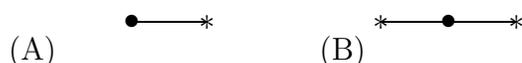


\begin{thebibliography}{xx}


\bibitem{spinor}
{\sc L. Arenas-Carmona}, \textit{Applications of spinor class
fields: embeddings of orders and quaternionic lattices}, Ann.
Inst. Fourier \textbf{53} (2003), 2021--2038.

\bibitem{proj}
{\sc L. Arenas-Carmona}, \textit{Projective representations in algebras and
cohomology}, Arch. Math. \textbf{97}  (2011), 105-113.

\bibitem{abelianos}
{\sc L. Arenas-Carmona}. \textit{Representation fields for
commutative orders}, Ann.
Inst. Fourier \textbf{62} (2012), 807-819.

\bibitem{continuity}
{\sc L. Arenas-Carmona}, \textit{Maximal selectivity for orders in fields}, \textit{J. Number T.} \textbf{132},  (2012), 2748-2755.


\bibitem{Eichler2}
{\sc L. Arenas-Carmona}, \textit{Eichler orders, trees and Representation Fields},  Int. J. Number Th.
\textbf{9}  (2013), 1725-1741.

\bibitem{qgraphs}
{\sc L. Arenas-Carmona}, \textit{Computing quaternion quotient graphs via representations of orders}, J. Algebra \textbf{402}  (2014), 258-279.

\bibitem{Chan}
{\sc W.K. Chan} and {\sc F. Xu}, \textit{On representations of
spinor genera}, Compositio Math. \textbf{140.2} (2004), 287-300.

\bibitem{Chevalley}
{\sc C. Chevalley}, \textit{L'arithm\'etique sur les alg\`ebres de
matrices}, Herman, Paris, 1936.

\bibitem{Deuring}
{\sc M. Deuring}, \textit{Die Anzahl der Typen von Maximalordungen einer definiten Quaternionenalgebra 
mit primer Grundzahl}, Jahresbericht der Deutschen Mathematiker-Vereinigung \textbf{54} (1951), 24-41

\bibitem{earnest}
{\sc A.G. Earnest}, \textit{On the representation of integers with large square factors by positive define ternary quadratic forms}, Mathematika \textbf{31.2} (1984) 252-257.

\bibitem{FriedmannQ}
{\sc T. Chinburg} and {\sc E. Friedman}, \textit{An embedding
theorem for quaternion algebras}, J. London Math. Soc.
\textbf{60.2} (1999), 33-44.


\bibitem{Guo}
{\sc X. Guo} and {\sc H. Qin}, \textit{An embedding theorem for
Eichler orders}, J. Number Theory \textbf{107} (2004), 207-214.

\bibitem{ae}
{\sc J.S. Hsia} and {\sc M. J\"ochner}, \textit{Almost strong approximations for definite quadratic spaces},
Invent. Math. \textbf{129} (1997) 471-487.

\bibitem{Ichi}
{\sc J.-I. Igusa}, \textit{Class number of a definite quaternion with prime discriminant}
Proc. Natl. Acad. Sci. USA. \textbf{44.4} (1958) 312-314.

\bibitem{lino}
{\sc B. Linowitz}, \textit{Selectivity in quaternion algebras},
 J. Number Theory \textbf{132} (2012), 1425-1437.

\bibitem{Macla}
{\sc C. Maclachlan}, \textit{Optimal embeddings in quaternion
algebras}, J. Number Theory \textbf{128} (2008), 2852-2860.

\bibitem{pays}
{\sc I.  Pays}, \textit{Arbres, ordres maximaux et formes quadratiques enti\`eres}, London Math. Soc. Lectures Note
Ser. \textbf{215}, Cambridge Univ. Press., Cambridge (1995) 209-230.

\bibitem{isaa}
{\sc I.  Saavedra}, \textit{C\'alculos expl\'icitos de la imagen espinorial local relativa
para ´ordenes c\'iclicos}, Master Thesis, Universidad de Chile, Santiago, 2014.

\bibitem{serre80}
{\sc J.-P. Serre}, \textit{Trees}, Springer Verlag, Berlin, 1980.

\bibitem{vigneras}
{\sc M.-F. Vigneras}, \textit{Arithm\'etique des alg\`ebres de Quaternions}, Springer Verlag, Berlin,
1980.


\end{thebibliography}
\end{document}